\documentclass[]{amsart}
\usepackage[centertags]{amsmath}
\usepackage{amsfonts}
\usepackage{graphicx}
\usepackage{amssymb}
\usepackage{amsthm}
\theoremstyle{plain}
\newtheorem{theorem}{Theorem}[section]
\newtheorem{proposition}[theorem]{Proposition}
\newtheorem{definition}[theorem]{Definition}
\newtheorem{lemma}[theorem]{Lemma}
\newtheorem{corollary}[theorem]{Corollary}
\newtheorem{example}[theorem]{Example}

\begin{document}
\title[On some generalizations of BCC-algebras]
      {On some generalizations of BCC-algebras}

\author{Wieslaw A. Dudek}
\address{Institute of Mathematics and Computer Science, Wroclaw University of Technology\\
 Wyb. Wyspianskiego 27, 50-370 Wroclaw, Poland}
 \email{wieslaw.dudek@pwr.wroc.pl}
\author{Janus Thomys}
\address{Lister Meile 56, 30161 Hannover\\
 Germany}
\email{janus.thomys@htp-tel.de}

\begin{abstract}
We describe weak BCC-algebras (also called BZ-algebras) in which
the condition $(xy)z=(xz)y$ is satisfied only in the case when
elements $x,y$ belong to the same branch. We also characterize
branchwise commutative and branchwise implicative weak
BCC-algebras satisfying this condition. We also describe
connections between various types of implicative weak
BCC-algebras.

 \end{abstract}
 \maketitle

\section{Introduction}

Realistic simulation of human decision making process has been a goal for artificial intelligence development for decades. Decisions that are made based on both certain and uncertain types of information are a special focus and the logic behind those decisions is dominant in proof theory. Such logic is at the core of any system or a tool for applications, in both mathematics and computers. In addition to the classical logic, many logic systems that deal with various aspects of uncertainty of information (e.g. fuzziness, randomness, etc.) are built on top of it, such as many-valued logic or fuzzy logic. A real life example of such uncertainty may be incomparability of data. To deal with fuzzy and uncertain information, computer science relies heavily on non-classical logic.

In recent years, motivated by both theory and application, the
study of $t$-norm-based logic systems and the corresponding
pseudo-logic systems has become a great focus in the field of
logic. Here, $t$-norm-base algebraic investigations were first to
the corresponding algebraic investigations, and in the case of
pseudo-logic systems, algebraic development was first to the
corresponding logical development (see for example \cite{Ior}). It
is well known that BCK and BCI-algebras are inspired by some
implicational logic. This inspiration is illustrated by the
similarities of names. We have BCK-algebras and a BCK positive
logic, BCI-algebras and a BCI positive logic and so on. In many
cases, the connection between such algebras and their
corresponding logics is much stronger. In such cases one can give
a translation procedure which translates all well formed formulas
and all theorems of a given logic $\mathcal L$ into terms and
theorems of the corresponding algebra. Nevertheless
the study of algebras motivated by known logics is interesting and
very useful for corresponding logics also in the case when the
full inverse translation procedure is impossible.

To solve some problems on BCK-algebras
Y.Komori introduced in \cite{Ko} the new class of algebras called
BCC-algebras. In view of strongly connections with a
BIK$^+$-logic, BCC-algebras also are called BIK$^+$-algebras (cf.
\cite{Zh'07}) or BZ-algebras (cf. \cite{ZY}). Nowadays, the
mathematicians especially from China, Japan and Korea, have been
studying various generalizations of BCC-algebras. All these
algebras have one distinguished element and satisfy some common
identities. One of very important identities is the identity
$(xy)z=(xz)y$. This identity is satisfied in such algebras as
pre-logics (cf. \cite{Ch'02a}), Hilbert algebras and implication
algebras (cf. \cite{Ch'07}) strongly connected with MV-algebras
(cf. \cite{Ch'04}). This identity also holds in BCK-algebras and
some their generalizations, but not in BCC-algebras. BCC-algebras
satisfying this identity are BCK-algebras (cf. \cite{Du'90} or
\cite{Du'92}). The class of all bounded commutative BCC-algebras
is equivalent to the class of all MV-algebras (cf. \cite{Ch'07a}).

Therefore, it makes sense to consider such BCC-algebras and some
their generalizations for which this identity is satisfied only by
elements belonging to some subsets (cf. \cite{Du'10}).

On the other hand, many mathematicians investigate BCI-algebras in
which some basic properties are restricted to some subset called
branches. For weak BCC-algebras such study was initiated in
\cite{Du'10} and continued in \cite{BT}.

In this paper we describe branchwise commutative and branchwise
implicative weak BCC-algebras in which the condition $(xy)z=(xz)y$
is satisfied only in the case when elements $x,y$ belong to the
same branch. We also characterize branchwise commutative and
branchwise implicative weak BCC-algebras satisfying this
condition. We also describe connections between various
generalizations of implicative weak BCC-algebras. Finally, we
consider weak BCC-algebras with condition $(S)$.

\section{Basic definitions and facts}
The BCC-operation will be denoted by juxtaposition. Dots will be
used only to avoid repetitions of brackets. For example, the
formula $((xy)(zy))(xz)=0$ will be written in the abbreviated form
as $(xy\cdot zy)\cdot xz=0$.

\begin{definition}\label{D-11.1}\rm A {\it weak BCC-algebra} is a system $(G;\cdot,0)$ of type $(2,0)$
satisfying the following axioms:
\begin{enumerate}
\item[$(i)$] \ $(xy\cdot zy)\cdot xz= 0$,
\item[$(ii)$] \ $xx = 0$,
\item[$(iii)$] \ $x0 = x$,
\item[$(iv)$] \ $xy = yx = 0\Longrightarrow x = y$.
\end{enumerate}
\end{definition}

By many mathematicians, especially from China and Korea, weak
BCC-algebras are called {\it BZ-algebras} (see for example \cite{ZY}), but we save the first name because it coincides
with the general concept of names used for algebras inspired by
various logics.

A weak BCC-algebra satisfying the identity
\begin{enumerate}
\item[$(v)$] \ $0x = 0$
\end{enumerate}
is called a {\it BCC-algebra}. A BCC-algebra with the condition
\begin{enumerate}
\item[$(vi)$] \ $(x\cdot xy)y = 0$
\end{enumerate}
is called a {\it BCK-algebra}.

An algebra $(G;\cdot,0)$ of type $(2,0)$ satisfying the axioms
$(i)$, $(ii)$, $(iii)$, $(iv)$ and $(vi)$ is called a {\it
BCI-algebra}. A weak BCC-algebra is a BCI-algebra if and only if
it satisfies the identity $xy\cdot z = xz\cdot y$ (cf.
\cite{Du'90}).

A weak BCC-algebra which is not a BCC-algebra is called {\it
proper} if it is not a BCI-algebra. A proper weak BCC-algebra has
at least four elements. But there are only two such non-isomorphic
weak BCC-algebras (see \cite{Du'96}). 

In any weak BCC-algebra we can define a natural partial order
$\leqslant$ putting
\begin{equation}\label{e1}
x\leqslant y \Longleftrightarrow xy = 0.
\end{equation}
This means that a weak BCC-algebra can be considered as a
partially ordered set with some additional properties.

\begin{proposition}\label{P-11.2} An algebra $(G;\cdot,0)$ of type $(2,0)$ with a relation
$\leqslant$ defined by \eqref{e1} is a weak BCC-algebra if and
only if for all $x,y,z\in G$ the following conditions are
satisfied:
\begin{enumerate}
\item[$(i')$] \ $xy\cdot zy\leqslant xz$,
\item[$(ii')$] \ $x\leqslant x$,
\item[$(iii')$] \ $x0 = x$,
\item[$(iv')$] \ $x\leqslant y$ and $y\leqslant x$ imply $x =
y$.\hfill $\Box{}$
\end{enumerate}
\end{proposition}

From $(i')$ it follows that in weak BCC-algebras implications
\begin{equation} \label{e2}
x \leqslant y \Longrightarrow xz \leqslant yz
\end{equation}
\begin{equation} \label{e3}
x \leqslant y \Longrightarrow zy \leqslant zx
\end{equation}
are valid for all $x,y,z \in G$.

In the investigations of algebras connected with various types of
logics an important role plays the self map $\varphi(x)=0x$. This
map was formally introduced in \cite{DT} for BCH-algebras, but
earlier it was used in \cite{Du'86} and \cite{Du'88} to
investigate some classes of BCI-algebras. Later, in \cite{DZW}, it
was used to characterize some ideals of weak BCC-algebras. Recall
that a subset $A$ weak BCC-algebra is called a {\it BCC-ideal} if
$0\in A$, and for all $y\in A$ from $xy\cdot z\in A$ it follows
$xz\in A$. A special case of a BCC-ideal  is a {\it BCK-ideal},
i.e., a subset $A$ such that $0\in A$, and  $y,xy\in A$ imply
$x\in A$. In the literature BCK-ideals also are called {\it ideals}.

The main properties of this map are collected in the following
theorem proved in \cite{DZW}.

\begin{theorem}\label{T-fi} Let $G$ be a weak BCC-algebra. Then
\begin{enumerate}
\item[$(1)$] \ $\varphi^{2}(x) \leqslant x$,
\item[$(2)$] \ $x\leqslant y\Longrightarrow\varphi(x)=\varphi(y)$,
\item[$(3)$] \ $\varphi^{3}(x) = \varphi(x)$,
\item[$(4)$] \ $\varphi^{2}(xy) = \varphi^{2}(x)\varphi^{2}(y)$
\end{enumerate}
for all $x,y\in G$. \hfill$\Box{}$
\end{theorem}

A weak BCC-algebra in which Ker$\varphi(x)=\{0\}$ is called {\it
group-like} or {\it anti-grouped}. A weak BCC-algebra
$(G;\cdot,0)$ is group-like if and only if there exists a group
$(G;*,0)$ such that $xy=x*y^{-1}$ (cf. \cite{Du'88}, \cite{DKB} or
\cite{ZY}).

The set of all minimal (with respect to $\leqslant$) elements of
$G$ will be denoted by $I(G)$. It is a subalgebra of $G$.
Moreover,
$$
I(G)=\varphi(G)=\{a\in G:\varphi^2(a)=a\}
$$
(cf. \cite{DKB}). The set
\[
B(a)=\{x\in G\,:\,a\leqslant x\},
\]
where $a\in I(G)$ is called the {\it branch} initiated by $a$.
Branches initiated by different elements are disjoint (cf.
\cite{DZW}). Comparable elements are in the same branch, but there
are weak BCC-algebras containing branches in which not all
elements are comparable.

\begin{lemma}\label{L-21} {\rm (cf. \cite{DKB})}
Elements $x$ and $y$ are in the same branch if and only if $xy\in
B(0)$.\hfill$\Box{}$
\end{lemma}

\begin{lemma}\label{L-22} {\rm (cf. \cite{DKB})}
$B(0)$ is a subalgebra of $G$. It is a maximal BCC-algebra
contained in $G$. \hfill$\Box{}$
\end{lemma}
\begin{lemma}\label{L-23} {\rm (cf. \cite{BT})} In a weak BCC-algebra $G$ for all $a,b\in I(G)$ we have
$B(a)B(b)=B(ab)$. \hfill$\Box{}$
\end{lemma}

The identity
\begin{equation}\label{e-sol}
 xy\cdot z=xz\cdot y.
\end{equation}
plays an important role in the theory of BCI-algebras.
It is used in the proofs of many basic facts.

\begin{definition}\rm
A weak BCC-algebra $G$ is called {\em solid} if the above
condition is valid for all $x$, $y$ belonging to the same branch
and arbitrary $z\in G$.
\end{definition}

A simple example of a solid weak BCC-algebra is a BCI-algebra.
Also BCK-algebra is a solid weak BCC-algebra. A solid BCC-algebra
is a BCK-algebra. But there are solid weak BCC-algebras which are
not BCI-algebras. The smallest such weak BCC-algebra has $5$
elements (cf. \cite{Du'10}).

\begin{example}\label{Ex-25}\rm Consider the set $X=\{0,1,2,3,4,5\}$ with the operation defined by the following
table:

\begin{center}$
\begin{array}{c|cccccc} \cdot&0&1&2&3&4&5\\ \hline\rule{0pt}{11pt}
0&0&0&4&4&2&2\\
1&1&0&4&4&2&2\\
2&2&2&0&0&4&4\\
3&3&2&1&0&4&4\\
4&4&4&2&2&0&0\\
5&5&4&3&3&1&0
\end{array}$\\[3mm]
\end{center}
Since $(S;*,0)$, where $S=\{0,1,2,3,4\}$, is a BCI-algebra (see
\cite{Huang}), it is not difficult to verify that
$(X;\cdot,0)$ is a weak BCC-algebra. It is proper because $(5\cdot 3)\cdot 2\ne
(5\cdot 2)\cdot 3$. Simple calculations show that this weak BCC-algebra is
solid. \hfill$\Box{}$
\end{example}

\begin{proposition}\label{P26} {\rm (cf. \cite{Du'10})}
In any solid weak BCC-algebra we have

$(a)$ \ $x\cdot xy\leqslant y$,

$(b)$ \ $x(x\cdot xy)=xy$

\noindent for all $x,y$ belonging to the same branch.
\hfill$\Box{}$
\end{proposition}

\begin{corollary}\label{C27}
In a solid weak BCC-algebra from $x,y\in B(a)$ it follows $x\cdot
xy,\,y\cdot yx\in B(a)$. \hfill$\Box{}$
\end{corollary}

\begin{proposition}\label{P-211} {\rm (cf. \cite{BT})}
The map $\varphi(x)=0x$ is an endomorphism of each solid weak
BCC-algebra. \hfill$\Box{}$
\end{proposition}

\begin{proposition}\label{P-212} In each solid weak BCC-algebra
for $xy$, $xz$ $($or $xy$ and $zy )$ belonging to the same branch
we have $xy\cdot xz\leqslant zy$.
\end{proposition}
\begin{proof} Indeed, $(xy\cdot xz)\cdot zy=(xy\cdot zy)\cdot
xz=0$.
\end{proof}

\section{Commutative solid weak BCC-algebras}

In any BCK-algebra $G$ we can define a binary operation $\wedge$
by putting
$$
x\wedge y= y\cdot yx
$$
for all $x,y\in G$. A BCK-algebra satisfying the identity
\begin{equation}\label{e5}
x\cdot xy=y\cdot yx,
\end{equation}
i.e., $y\wedge x=x\wedge y$, is called {\it commutative}. A
commutative BCK-algebra is a lower semilattice with respect to the
operation $\wedge$.

This definition cannot be extended to BCI-algebras, BCC-algebras
and weak BCC-algebras since in any weak BCC-algebra satisfying
\eqref{e5} we have $0\cdot 0x= x\cdot x0=0$, i.e.,
$\varphi^2(x)=0$ for every $x\in G$. Thus
$\varphi(x)=\varphi^3(x)=0$, by Theorem \ref{T-fi}. Hence in this
algebra $0\leqslant x$ for every $x\in G$.  This means that this
algebra is a commutative BCC-algebra. But in any BCC-algebra $G$
we have $0\leqslant yx$ for all $x,y\in G$, which together with
\eqref{e3} implies $y\cdot yx\leqslant y$. Thus a commutative
BCC-algebra satisfies the inequality
$$x\cdot xy=y\cdot yx\leqslant y.
$$
Consequently, it satisfies the identity $(x\cdot xy)y=0$, so it is
a BCK-algebra. Hence a commutative (weak) BCC-algebra is a
commutative BCK-algebra. Analogously, a commutative BCI-algebra is
a commutative BCK-algebra.

But there are weak BCC-algebras in which the condition \eqref{e5}
is satisfied only by elements belonging to the same branch.

\begin{example}\label{Ex-31}\rm A weak BCC-algebra defined by the following table
$$
\begin{array}{c|cccccc} \cdot&0&a&b&c&d\\ \hline
\rule{0pt}{12pt}0&0&0&0&c&c\\
a&a&0&0&c&c\\
b&b&a&0&d&c\\
c&c&c&c&0&0\\
d&d&c&c&a&0
\end{array}
$$
has two branches: $B(0)=\{0,a,b\}$ and $B(c)=\{c,d\}$. It is not
difficult to verify that in this weak BCC-algebra \eqref{e5} is
satisfied only by elements belonging to the same branch.
\hfill$\Box{}$
\end{example}

\begin{definition}\label{D-32}\rm A weak BCC-algebra in which \eqref{e5} is satisfied by
elements belonging to the same branch is called {\it branchwise
commutative}.
\end{definition}

\begin{theorem}{\rm (cf. \cite{Du'10})}\label{T-33} For a solid weak BCC-algebra $G$ the following conditions are
equivalent:
\begin{enumerate}
\item[$(1)$] \ $G$ is branchwise commutative,
\item[$(2)$] \ $xy=x(y\cdot yx)$ \ for $x,y$ from the same branch,
\item[$(3)$] \ $x=y\cdot yx$ \ for \ $x\leqslant y$,
\item[$(4)$] \ $x\cdot xy = y(y(x\cdot xy))$ \ for $x,y$ from the same branch,
\item[$(5)$] \ each branch of $G$ is a semilattice with respect to the operation
$x\wedge y=y\cdot yx$. \hfill$\Box{}$
\end{enumerate}
\end{theorem}

In the proof of the next theorem we will need the following
well-known result from the theory of BCK-algebras.

\begin{lemma} \label{L34} If $p$ is the greatest element of a commutative BCK-algebra $G$, then
$(G; \leqslant)$ is a distributive lattice with respect to the
operations $x\wedge y = y\cdot yx$ and $x\vee y=p(px\wedge py)$.
\hfill$\Box{}$
\end{lemma}

\begin{theorem} \label{T35} In a solid branchwise commutative weak BCC-algebra $G$, for every $p\in G$, the set $A(p)=\{x\in G:x\leqslant p\}$ is a distributive lattice with respect to the operations $x \wedge  y = y\cdot yx$ and $x \vee_{p} y = p(px\wedge py)$.
\end{theorem}
\begin{proof} (A). We prove that $(A^{p};\leqslant )$, where
$$
A^{p} = \{px: x\in A(p)\},
$$
is a distributive lattice.

First, we show that $A^p$ is a subalgebra of $B(0)$.
 It is clear that $0=pp\in A^p$ and $A(p)\subseteq B(a)$ for some $a\in I(G)$. Thus
$A^p\subset B(0)$. Obviously $a\leqslant x$ for every $x\in A(p)$.
Consequently, $px\leqslant pa$. Hence $pa$ is the greatest element
of $A^p$.

Let $px$, $py$ be arbitrary elements of $A^p$. Then obviously
$yx\in B(0)$ and $z=p\cdot yx\in A^p$ because $zp = (p\cdot yx)p =
0$. Moreover, $zx = (p\cdot yx)x = px\cdot yx \leqslant py$, by
$(i')$. Since
$$
(px\cdot pz)\cdot zx=(px\cdot zx)\cdot pz=0,
$$
we also have $px\cdot pz\leqslant zx\leqslant py$. Therefore,
$0=(px\cdot pz)\cdot py=(px\cdot py)\cdot pz$, i.e.,
\begin{equation} \label{e6}
px\cdot py\leqslant pz.
\end{equation}

On the other hand, since a weak BCC-algebra $G$ is branchwise
commutative, for every $y\in A(p)$, according to Theorem
\ref{T-33}, we have $p\cdot py = y$. Hence $px\cdot py=(p\cdot
py)x=yx$. But $pz\cdot yx=p(p\cdot yx)\cdot yx=(p\cdot yx)(p\cdot
yx)=0$. Thus $pz\leqslant yx = px\cdot py$, which together with
\eqref{e6} gives
$$
px\cdot py=pz.
$$
Hence $A^p$ is a subalgebra of $B(0)$. Obviously, $B(0)$ as a
BCC-algebra contained in $G$ is commutative, and consequently it
is a commutative BCK-algebra. Thus $A^p$ is a commutative
BCK-algebra, too. By Lemma \ref{L34}, $(A^p;\leqslant)$ is a
distributive lattice.

\medskip

(B). Now we show that $(A(p);\leqslant)$ is a distributive
lattice. Clearly, $p$ is the greatest element of $A(p)$.

Let  $x,y \in A(p)$. Then $px,py\in A^p$ and from the fact that
$(A^{p};\leqslant)$ is a lattice it follows that there exists the
last upper bound $pz\in A^p$, i.e.,
\begin{equation}\label{e7}
px \vee_{p} py = pz.
\end{equation}
Observe that for $x,y\in A(p)$ we have
\begin{equation}\label{e8}
py\leqslant px\Longleftrightarrow x\leqslant y.
\end{equation}
Indeed, in view of \eqref{e3}, $x\leqslant y$ implies $py\leqslant
px$. Similarly, $py\leqslant px$ implies $p\cdot px\leqslant
p\cdot py$. But $G$ is branchwise commutative, hence by Theorem
\ref{T-33},  for every $v\in A(p)$ we have $p\cdot pv=v$.
Therefore $x = p\cdot px\leqslant p\cdot py=y$.

From \eqref{e7} and \eqref{e8} it follows that $z$ is the greatest
lower bound for $x$ and $y$. Hence $x \wedge y = z$. Moreover, we
have $p(x \wedge y) = pz$ and $px \vee_{p} py = pz$, which implies
\begin{equation} \label{e9}
p(x \wedge y) = px \vee_{p} py.
\end{equation}

Analogously, we can prove that for all $x,y\in A(p)$ there exists
$x\vee_p y$ and
\begin{equation} \label{e10}
p(x \vee_{p} y) = px \wedge py.
\end{equation}
Therefore $(A(p); \leqslant )$ is a lattice.

Since \eqref{e9} and \eqref{e10} are satisfied in $A^p$ and
$(A^p;\leqslant)$ is a distributive lattice, we have
$$
px \vee_{p} (py\wedge pz) = (px\vee_{p} py)\wedge (px\vee_{p} pz)
$$
for all $px,py,pz\in A^p$.

This, in view of \eqref{e9} and \eqref{e10}, gives
$$
p(x\wedge (y\vee_{p}z)) = p((x\wedge y)\vee_{p} (x\wedge z)).
$$
Consequently,
$$
p\cdot p(x\wedge (y\vee_{p} z)) = p\cdot p((x\wedge
y)\vee_{p}(x\wedge z))
$$
and
$$
x\wedge (y\vee_{p} z) = (x\wedge y)\vee_{p} (x\wedge z),
$$
because $x\wedge (y\vee_{p} z),\,(x\wedge y)\vee_{p} (x\wedge
z)\in A(p)$. This means that $(A(p);\leqslant)$ is a distributive
lattice.

In this lattice $x\vee_{p}y=p(p(x\vee_{p}y))=p(px\wedge py)$.

This completes the proof.
\end{proof}

\begin{definition} \label{D-36} \rm A weak BCC-algebra $G$ is called {\it restricted}, if every its branch has the
greatest element.
\end{definition}

The greatest element of the branch $B(a)$ will be denoted by
$1_a$. By $N_a$ we will denote the unary operation $N_a:G\to G$
defined by $N_ax=1_ax$.

\begin{lemma} \label{L-37} The main properties of the operation $N_a$ in restricted solid weak BCC-algebras are as follows:
\begin{enumerate}
\item[$(1)$] \ $N_{a}1_{a} =0$ and $N_{a}0 = 1_{a}$,
\item[$(2)$] \ $N_{a}N_{a}x \leqslant x$,
\item[$(3)$] \ $(N_{a}x)y = (N_{a}y)x$,
\item[$(4)$] \ $x \leqslant y\Longrightarrow N_{a}y\leqslant N_{a}x$,
\item[$(5)$] \ $N_{a}xN_{a}y \leqslant yx$,
\item[$(6)$] \ $N_{a}N_{a}N_{a}x = N_{a}x$,
\end{enumerate}
where $x,y\in B(a)$.\hfill$\Box{}$
\end{lemma}

\begin{definition} \label{D-38} \rm A restricted weak BCC-algebra $G$ is called {\it involutory}, if $N_aN_ax=x$
holds for every $x\in B(a)$. An element $x$ satisfying this
condition is called an {\it involution}.
\end{definition}

A simple example of involutions in restricted solid weak
BCC-algebras are $a\in I(G)$ and $1_a$.

In an involutory weak BCC-algebra the map $N_a:B(a)\to B(0)$ is
one-to-one. Thus in an involutory weak BCC-algebra with finite
$B(0)$ all branches are finite.

\begin{proposition} \label{P-39} Any branchwise commutative restricted solid weak BCC-algebra is involutory.
\end{proposition}
\begin{proof} Indeed,  $N_{a}N_{a}x = 1_{a}(1_{a}x) = x(x1_{a}) = x0 = x$ for every $x\in B(a)$.
\end{proof}

\begin{lemma} \label{L-310} In an involutory solid weak
BCC-algebra
$$
xy=N_{a}yN_{a}x
$$
is valid for all $a\in I(G)$ and $x,y\in B(a)$.
\end{lemma}
\begin{proof}
In fact, $xy=(N_{a}N_{a}x)y=(1_a\cdot 1_ax)y=1_ay\cdot 1_ax=
N_{a}yN_{a}x$.
\end{proof}

\begin{proposition}\label{P-311}
A solid weak BCC-algebra is involutory if and only if
$$
xN_ay=yN_ax
$$
holds for all $a\in I(G)$ and $x,y\in B(a)$.
\end{proposition}
\begin{proof}
Clearly $N_ax,N_ay\in B(0)$ for $x,y\in B(a)$. Thus $yN_ax,
xN_ay\in B(a)$, and consequently
$$
xN_ay\cdot yN_ax=(x\cdot 1_ay)(y\cdot 1_ax)=x(y\cdot 1_ax)\cdot
1_ay=(1_a\cdot 1_ax)(y\cdot 1_ax)\cdot 1_ay.
$$
Since $(1_a\cdot 1_ax)(y\cdot 1_ax)\leqslant 1_ay$, by $(i')$,
from \eqref{e2} it follows
$$
(1_a\cdot 1_ax)(y\cdot 1_ax)\cdot 1_ay\leqslant 1_ay\cdot 1_ay=0.
$$
Hence $xN_ay\cdot yN_ax=0$. Analogously we show $yN_ax\cdot
xN_ay=0$, which by $(iv)$ implies $xN_ay=yN_ax$.

On the other hand, if $xN_ay=yN_ax$ holds for all $a\in I(G)$ and
$x,y\in B(a)$, then for $y=1_a$ and arbitrary $x\in B(a)$ we have
$$
x=x\cdot 1_a1_a=xN_a1_a=1_aN_ax=N_aN_ax,
$$
which means that this weak BCC-algebra is involutory.
\end{proof}

\begin{theorem} \label{T-311} For an involutory solid weak BCC-algebra $G$ the following statements are equivalent:
\begin{enumerate}
\item[$(1)$] \ each branch of  $G$ is a lower semilattice,
\item[$(2)$] \ each branch of $G$ is a lattice.
\end{enumerate}
Moreover, if $\,(B(a);\leqslant)$ is a lattice, then

\medskip
\centerline{$x\wedge y=N_{a}(N_{a}x\vee_{a}N_{a}y)$ \ and \
$x\vee_{a}y=N_{a}(N_{a}x\wedge N_{a}y)$.}
\end{theorem}
\begin{proof}$(1)\Longrightarrow (2)$ \ Since each branch $B(a)$ of $G$ is a lower semilattice, then $N_{a}x\wedge N_{a}y$ exists for all $x,y\in B(a)$, i.e., for all $N_ax,N_ay\in B(0)$. Hence $N_{a}x\wedge N_{a}y\leqslant N_{a}x$, which gives $zN_{a}x\leqslant z(N_{a}x\wedge  N_{a}y)$ for each $z\in B(a)$. Similarly, $zN_{a}y \leqslant z(N_{a}x \wedge  N_{a}y)$. This means that $z(N_{a}x\wedge  N_{a}y)$ is an upper bound of $zN_{a}x$ and $zN_{a}y$. Let us assume that $u\in B(a)$ is another upper bound for $zN_{a}x$ and $zN_{a}y$. From $zN_{a}x \leqslant u$ and $zN_{a}y\leqslant u$ we have $zu\leqslant z(zN_{a}x)$ and $zu\leqslant z(zN_{a}y)$. But $z(zN_av)\leqslant N_av$ for $v,z\in B(a)$, because $N_av\in B(0)$ and $zN_av\in B(a)$. Hence $zu\leqslant N_ax$ and $zu\leqslant N_ay$, which implies $zu \leqslant N_{a}x \wedge  N_{a}y$.  Thus $z(N_{a}x \wedge  N_{a}y) \leqslant z(zu)\leqslant u$ and $z(N_{a}x \wedge  N_{a}y)$ is the least upper bound of $zN_{a}x$ and $zN_{a}y$. Therefore for every $x,y,z\in B(a)$  there exists the least upper bound of $zN_{a}x$ and $zN_{a}y$, i.e., $zN_ax\vee_a zN_ay$. In particular, for every $x,y\in B(a)$ there exists
$$1_aN_ax\vee_a 1_aN_ay=N_aN_ax\vee_a N_aN_ay=x\vee_a y.
$$
This shows that $(B(a); \leqslant)$ is an upper semilattice.
Consequently, $B(a)$ is a lattice.

$(2)\Longrightarrow (1)$ \ Obvious.

Since $(B(a);\leqslant)$ is a lattice for every $a\in I(G)$, using
the same argumentation as in the second part of the proof of
Theorem \ref{T35} we can show that $N_a(x\wedge y)=N_ax\vee_{a} N_ay$
for $x,y\in B(a)$. Thus,
$$ x\wedge y=N_{a}N_{a}(x\wedge y)=N_{a}(N_{a}x\vee_{a}N_{a}y).
$$
Analogously, $N_a(x\vee_a y)=N_ax\wedge N_ay$ implies
$$
x\vee_{a}y=N_{a}N_{a}(x\vee_{a}y)= N_{a}(N_{a}x\wedge N_{a}y).
$$
This completes the proof.
\end{proof}

\section{n-fold branchwise commutative weak BCC-algebras}

In a weak BCC-algebra $G$ for all $x,y\in G$ we put $xy^{0}=x$ and
$xy^{n+1} = (xy^{n})y$ for any non-negative integer $n$.

\begin{definition} \label{D-41} \rm A weak BCC-algebra G is called {\it $n$-fold branchwise commutative} (shortly: {\it $n$-b commutative}), if there exists a natural number $n$ such that
\begin{equation}\label{n-b}
xy = x(y\cdot yx^{n})
\end{equation}
holds for $x,y$ belonging to the same branch.
\end{definition}

From Theorem \ref{T-33} it follows that for $n=1$ it is an
ordinary branchwise commutative weak BCC-algebra.

\begin{theorem} \label{T-42} For a solid weak BCC-algebra $G$ the following conditions are equivalent:
\begin{enumerate}
\item[$(a)$] \ $G$ is $n$-b commutative,
\item[$(b)$] \ $x\cdot xy\leqslant y\cdot yx^{n}$ \ for $x,y$ belonging to the same branch,
\item[$(c)$] \ $x\leqslant y\Longrightarrow x\leqslant y\cdot yx^{n}$.
\end{enumerate}
\end{theorem}
\begin{proof} $(a)\Longrightarrow (b)$ \ Let $x,y\in B(a)$ for some $a\in B(a)$. Then $xy\in B(0)$ and consequently
$x(y\cdot yx^n)\in B(0)$. So, $x$ and $y\cdot yx^n$ are in the
same branch. Hence $y\cdot yx^n\in B(a)$ and
$$
(x\cdot xy)(y\cdot yx^n)=x(y\cdot yx^n)\cdot xy=xy\cdot xy=0.
$$
Therefore $x\cdot xy\leqslant y\cdot yx^{n}$.

$(b)\Longrightarrow (c)$ \ Obvious.

$(c)\Longrightarrow (a)$ \ Since $0\leqslant xy$ for $x,y\in
B(a)$, we have
\begin{equation}\label{nn}
x\cdot xy\leqslant x.
\end{equation}
Consequently, $yx\leqslant y(x\cdot xy)$. This implies
$$
yx\cdot x\leqslant y(x\cdot xy)\cdot x=yx\cdot (x\cdot xy)
$$
and
$$
yx\cdot (x\cdot xy)\leqslant y(x\cdot xy)\cdot (x\cdot xy).
$$
Therefore
$$
yx^{2}=yx\cdot x\leqslant y(x\cdot xy)\cdot (x\cdot xy)=y(x\cdot
xy)^{2},
$$
i.e.,
$$
yx^2\leqslant y(x\cdot xy)^2.
$$
From this inequality we obtain
\begin{equation}\label{nnn}
yx^2\cdot x\leqslant y(x\cdot xy)^2\cdot x.
\end{equation}
From \eqref{nn} we get
$$
y(x\cdot xy)^2\cdot x\leqslant y(x\cdot xy)^2\cdot (x\cdot xy),
$$
which together with \eqref{nnn} gives
$$
yx^3\leqslant y(x\cdot xy)^3.
$$
Repeating the above procedure we can see that
$$
yx^{n}\leqslant y(x\cdot xy)^{n}
$$
holds for every natural $n$. Hence
\begin{equation} \label{nnnn}
x(y\cdot yx^{n})\leqslant x(y\cdot y(x\cdot xy)^{n}).
\end{equation}

Obviously $x\cdot xy\leqslant y$ for $x,y$ belonging to the same
branch. Applying $(c)$ to the last inequality we obtain $x\cdot
xy\leqslant y\cdot y(x\cdot xy)^n$. Hence, by \eqref{e3} and
Proposition \ref{P26} $(b)$, we conclude $x(y\cdot y(x\cdot
xy)^{n})\leqslant x(x\cdot xy) = xy$. Consequently,
\begin{equation} \label{mm}
x(y\cdot y(x\cdot xy)^{n})\leqslant xy.
\end{equation}

Combining \eqref{nnnn} and \eqref{mm} we get
\begin{equation} \label{aaa}
x(y\cdot yx^{n})\leqslant xy.
\end{equation}
Thus $x(y\cdot yx^{n})\in B(0)$. This means that $y\cdot yx^{n}\in
B(a)$. But in this case $yx^n\in B(0)$. Indeed, if $yx^{n}\in
B(b)$ for some $b\in I(G)$, then $y\cdot yx^n\in B(a)B(b)=B(ab)$.
So, $y\cdot yx^n\in B(a)\cap B(ab)$. Thus $B(a) = B(ab)$, i.e., $a
= ab$. Hence $0=ab\cdot a=aa\cdot b=0b$. Therefore $b\in B(0)$ and
$b = 0$ because $b\in I(G)$.

This means that
$$
(y\cdot yx^n)y = 0\cdot yx^n=0.
$$
Consequently, $y\cdot yx^{n}\leqslant y$, which, by \eqref{e3},
implies $xy\leqslant x(y\cdot yx^{n})$.

Comparing the last inequality with \eqref{aaa} we obtain $xy =
x(y\cdot yx^{n})$. This completes the proof.
\end{proof}

\section{Implicative solid weak BCC-algebras}

Implicative and positive implicative BCC-algebras are originating
from the systems of positive implicational calculus and weak
positive implicational calculus in the implicational functor in
logical systems. In this section we will also deal with some
generalized implicative and positive implicative solid weak
BCC-algebras.

\begin{definition} \label{D-51} \rm A weak BCC-algebra $G$ is called {\it branchwise implicative}, if
$$
x\cdot yx=x
$$
holds for all $x,y$ belonging to the same branch of $G$.
\end{definition}

\begin{theorem} \label{T-52}{\rm (\cite{Du'10}, Theorem 3.8)} Any solid branchwise implicative weak BCC-algebra
is branchwise commutative. \hfill$\Box{}$
\end{theorem}

\begin{theorem} \label{T-53} In a solid branchwise implicative weak BCC-algebra the equation
$$
xy\cdot 0y=(xy\cdot 0y)y\cdot 0y
$$
is satisfied by all elements belonging to the same branch.
\end{theorem}
\begin{proof} Let $G$ be branchwise implicative and solid. Then, according to $(i)$, for all $x,y\in B(a)$,
we have
\begin{equation} \label{2.5.1}
(xy\cdot 0y)x = 0.
\end{equation}
So, $xy\cdot 0y\in B(a)$ and $xy,\,x(xy\cdot 0y)\in B(0)$.
Therefore,
$$
\rule{2mm}{0mm}\arraycolsep=.5mm\begin{array}{rl}
 (xy\cdot 0y)\cdot x(xy\cdot 0y)&=(xy\cdot x(xy\cdot 0y))\cdot 0y\\[3mm]
 &=(x\cdot x(xy\cdot 0y))y\cdot 0y\\[3mm]
&= ((xy\cdot 0y)\cdot (xy\cdot 0y)x)y\cdot 0y \rule{14mm}{0mm}
{\rm by \ Theorem \ \ref{T-52}}\\[3mm]
&= ((xy\cdot 0y)0)y\cdot 0y \hfill {\rm by \ \eqref{2.5.1}}\\[3mm]
&= (xy\cdot 0y)y\cdot 0y.
\end{array}
$$
Hence
$$
(xy\cdot 0y)\cdot x(xy\cdot 0y)=(xy\cdot 0y)y\cdot 0y.
$$

Since $xy\cdot 0y$ and $x$ are in the same branch, the
implicativity shows that
$$
(xy\cdot 0y)\cdot x(xy\cdot 0y) = xy\cdot 0y,
$$
which together with the previous equation gives \ $xy\cdot
0y=(xy\cdot 0y)y\cdot 0y$.
\end{proof}

\begin{lemma}\label{L-54} In a solid weak BCC-algebra for $x,y$ belonging to the same branch holds
\begin{displaymath}
(xy\cdot 0y)y\cdot 0y \leqslant ((xy\cdot y)\cdot 0y)\cdot 0y.
\end{displaymath}
\end{lemma}
\begin{proof} Indeed, if $x,y\in B(a)$, then as in the previous proof we can also see that $xy\cdot 0y\in
B(a)$. Hence
$$
\rule{2mm}{0mm}\arraycolsep=.5mm\begin{array}{rl} ((xy\cdot
0y)y\cdot 0y)\cdot (((xy\cdot y)\cdot 0y)\cdot 0y)&
\leqslant  (xy\cdot 0y)y\cdot ((xy\cdot y)\cdot 0y)\rule{10mm}{0mm} {\rm by } \ (i')\\[3mm]
&= ((xy\cdot 0y)((xy\cdot y)\cdot 0y))\cdot y\\[3mm]
&\leqslant (xy\cdot (xy\cdot y))\cdot y\hfill {\rm by } \ (i')\\[3mm]
&\leqslant (x\cdot xy)\cdot y= xy\cdot xy= 0,
\end{array}
$$
 i.e.,
$$
((xy\cdot 0y)y\cdot 0y)\cdot (((xy\cdot y)\cdot 0y)\cdot 0y)=0.
$$
This implies
\begin{displaymath} ((xy\cdot 0y)y)\cdot 0y\leqslant
((xy\cdot y)\cdot 0y)\cdot 0y.
\end{displaymath}
The proof is complete.
\end{proof}

\begin{theorem} \label{T-55} If $\,I(G)$ is a BCK-ideal of a branchwise commutative solid weak BCC-algebra $G$
and
\begin{equation} \label{2.5.2}
xy\cdot 0y=((xy\cdot y)\cdot 0y)\cdot 0y
\end{equation}
is valid for all $x,y$ belonging to the same branch, then $G$ is
branchwise implicative.
\end{theorem}
\begin{proof} Let $x,y\in B(a)$ for some $a\in I(G)$. Then
$$
\rule{2mm}{0mm}\arraycolsep=.5mm\begin{array}{rl} x(x\cdot
yx)\cdot 0x&=(x(x\cdot
yx)\cdot 0)\cdot 0x\\[3mm]
&=(x(x\cdot yx)\cdot (x\cdot yx)(x\cdot
yx))\cdot 0x\\[3mm]
&=(x(x\cdot yx)\cdot (x(x\cdot yx)\cdot yx))\cdot 0x\rule{12mm}{0mm}{\rm since } \ x,\,x\cdot yx\in B(a)\\[3mm]
&= (yx\cdot (yx\cdot x(x\cdot yx)))\cdot 0x,
\end{array}
$$
because $yx,\,x(x\cdot yx)\in B(0)$ and $G$ is branchwise
commutative. But $B(0)$ is a subalgebra of $G$ (Lemma \ref{L-22}),
hence $yx\cdot x(x\cdot yx)\in B(0)$. Therefore
$$
\arraycolsep=.5mm\begin{array}{rl}
 (yx\cdot (yx\cdot x(x\cdot yx)))\cdot 0x&=(yx\cdot
0x)\cdot (yx\cdot x(x\cdot yx))\\[3mm]
&= (((yx\cdot x)\cdot 0x)\cdot 0x)\cdot (yx\cdot x(x\cdot yx)),
\end{array}$$
by \eqref{2.5.2}.

Since $(yx\cdot x)\cdot 0x\leqslant yx\cdot 0=yx\in B(0)$ implies
$(yx\cdot x)\cdot 0x\in B(0)$, from the above we obtain
$$
\arraycolsep=.5mm\begin{array}{rl}
 x(x\cdot yx)\cdot 0x
&=(((yx\cdot x)\cdot 0x)\cdot 0x)\cdot (yx\cdot x(x\cdot
yx))\\[3mm]
&=(((yx\cdot x)\cdot 0x)\cdot(yx\cdot x(x\cdot yx)))\cdot
0x\\[3mm]
&=(((yx\cdot x)\cdot (yx\cdot x(x\cdot yx)))\cdot 0x)\cdot 0x,
\end{array}
$$
because, as it is not difficult to see, $yx\cdot x,0x \in B(0a)$.

Now from the fact that $yx\cdot x(x\cdot yx)$ and $x(x\cdot yx)$
are in $B(0)$ and $G$ is branchwise commutative we have
$$
\arraycolsep=.5mm\begin{array}{rl}
 x(x\cdot yx)\cdot 0x
&=(((yx\cdot x)\cdot (yx\cdot x(x\cdot yx)))\cdot 0x)\cdot
0x\\[3mm]
&= (((yx\cdot (yx\cdot x(x\cdot yx)))\cdot x)\cdot 0x)\cdot 0x\\[3mm]
&=(((x(x\cdot yx)\cdot (x(x\cdot yx)\cdot yx))\cdot x)\cdot
0x)\cdot 0x.
\end{array}
$$
From this, in view of $x\cdot yx\in B(a)$, we get
$$
\rule{12mm}{0mm}\arraycolsep=.5mm\begin{array}{rl} x(x\cdot
yx)\cdot 0x
&= (((x(x\cdot yx)\cdot ((x\cdot xy) (x\cdot yx)))\cdot x)\cdot 0x)\cdot 0x$ \hfill $\\[3mm]
&= (((x(x\cdot yx)\cdot 0)\cdot x)\cdot 0x)\cdot 0x\\[3mm]
&= ((x(x\cdot yx))x\cdot 0x)\cdot 0x\\[3mm]
&= ((xx\cdot (x\cdot yx))\cdot 0x)\cdot 0x\\[3mm]
&= ((0\cdot (x\cdot yx))\cdot 0x)\cdot 0x\\[3mm]
&= ((0x\cdot (0\cdot yx))\cdot 0x)\cdot 0x \rule{21mm}{0mm}{\rm by \ Proposition \ \ref{P-211}}\\[3mm]
&= ((0x\cdot 0)\cdot 0x)\cdot 0x\\[3mm]
&= (0x\cdot 0x)\cdot 0x= 0\cdot 0x\in I(G),
\end{array}
$$
because $I(G)=\varphi(G)$. Hence $x(x\cdot yx)\cdot 0x\in I(G)$.
Also $\,0x\in I(G)$. Since, by the assumption, $I(G)$ is a
BCK-ideal of $G$, we obtain $x(x\cdot yx)\in I(G)$. But $x(x\cdot
yx)\in B(0)$, so $x(x\cdot yx)\in I(G)\cap B(0)$. Thus $x(x\cdot
yx)=0$. This means that $x\leqslant x\cdot yx\leqslant x$.
Consequently, $x\cdot yx=x$. Therefore $G$ is branchwise
implicative. The proof is complete.
\end{proof}

The example presented below shows that in the last theorem the
assumption on $I(G)$ is essential.

\begin{example}\label{Ex-56}\rm Consider a weak BCC-algebra $G$ defined by the following table:
\[
\begin{array}{c|ccccc} \cdot&0&1&2&3&4\\ \hline\rule{0pt}{11pt}
0&0&0&0&3&3\\
1&1&0&0&3&3\\
2&2&1&0&4&4\\
3&3&3&3&0&0\\
4&4&3&3&1&0\\
\end{array}
\]
Because $(S;\cdot,0)$, where $S=\{0,1,3,4\}$, is a BCI-algebra
(see \cite{Huang}, p.337) to show that $G$ is a weak BCC is
sufficient to check the axiom $(i)$ in the case when at least one
of the elements $x,y,z$ is equal to $2$. Such defined weak
BCC-algebra is proper since $23\cdot 4\ne 24\cdot 3$. It also is
branchwise commutative and satisfies \eqref{2.5.2} but it is not
branchwise implicative. Obviously $I(G)$ is not a BCK-ideal of
$G$. \hfill $\Box{}$
\end{example}

\section{Positive implicative weak BCC-algebras}

As it is well-know a BCK-algebra is
called {\it positive implicative}, if it satisfies the identity
\begin{equation} \label{2.5.3}
xy\cdot y=xy.
\end{equation}
In BCK-algebras this identity is equivalent to
\begin{equation} \label{2.5.4}
xy\cdot z=xz\cdot yz.
\end{equation}
Positive implicative BCC-algebras can be defined in the same way
(cf. \cite{Du'90} or \cite{Du'92}), however weak BCC-algebras cannot
because by putting $x=y$ in \eqref{2.5.3} we obtain $0x=0$ for every
$x\in G$. This means that a weak BCC-algebra (as well as BCI-algebra)
satisfying \eqref{2.5.3} or \eqref{2.5.4} is a BCC-algebra.
Therefore positive implicative weak BCC-algebras and BCI-algebras
should be defined in another way. One way was proposed by J. Meng
and X. L. Xin in \cite{MX}. They defined a positive implicative
BCI-algebra as a BCI-algebra satisfying the identity $xy=(xy\cdot
y)\cdot 0y$. (Equivalent conditions one can find in \cite{MX} and
\cite{Huang}.) Using this definition it can be proved that a
BCI-algebra is implicative if and only if it is both positive
implicative and commutative. Unfortunately, in the proof of this
result a very important role plays the identity \eqref{e-sol}. So,
this proof can not be transferred to weak BCC-algebras. In
connection with this fact, W.A. Dudek introduced in \cite{Du'10}
the new class of positive implicative weak BCC-algebras called by
him {\em $\varphi$-implicative}.

\begin{definition}\label{D-56} \rm A weak BCC-algebra $G$ is called {\it $\varphi$-implicative},
if it satisfies the identity
\begin{equation} \label{2.5.5}
xy=xy\cdot y(0\cdot 0y),
\end{equation}
i.e.,
$$
xy=xy\cdot y\varphi^2(y).
$$
If \eqref{2.5.5} is satisfied only by elements belonging to the
same branch, then we say that this weak BCC-algebra is {\it
branchwise $\varphi$-implicative}.
\end{definition}

It is clear that in the case of BCC-algebras the conditions
\eqref{2.5.5} and \eqref{2.5.3} are equivalent. Thus a BCC-algebra
is $\varphi$-implicative if and only if it is positive
implicative. For BCI-algebras and weak BCC-algebras it is not
true. A group-like weak BCC-algebra determined by a group, i.e., a
weak BCC-algebra $(G;\cdot,0)$ with the operation $xy=x*y^{-1}$
where $(G;*,0)$ is a group, is a simple example of a
$\varphi$-implicative weak BCC-algebra which is not positive
implicative.

\begin{definition}\label{D-57}\rm A weak BCC-algebra $G$ is called {\it weakly positive
implicative}, if it satisfies the identity
\begin{equation}\label{2.5.6}
xy\cdot z = (xz\cdot z)\cdot yz.
\end{equation}
If \eqref{2.5.6} is satisfied only by elements belonging to the
same branch, then we say that this weak BCC-algebra is {\it
branchwise weakly positive implicative}.
\end{definition}

\begin{example}\label{E-58}\rm Routine and easy calculations show that a weak BCC-algebra defined by
the following table:\\[3mm]
\centerline{$\begin{array}{c|ccc} \cdot&0&a&b\\
\hline
0&0&0&b\rule{0mm}{4mm}\\
a&a&0&b\\
b&b&b&0\\
\end{array}$}

\noindent is weakly positive implicative. \hfill$\Box{}$
\end{example}

\begin{lemma} \label{L-59} A solid weakly positive implicative weak BCC-algebra $G$ satisfies
the identity
\begin{equation} \label{2.5.7}
xy=(xy\cdot y)\cdot 0y.
\end{equation}
\end{lemma}
\begin{proof} By putting $y=0$ in \eqref{2.5.6} we obtain the identity $xz = (xz\cdot z)\cdot 0z$,
which is equivalent to \eqref{2.5.7}.
\end{proof}

\begin{theorem}\label{T-510} A solid weakly positive implicative weak BCC-algebra is branchwise
$\varphi$-implicative.
\end{theorem}
\begin{proof} Let $G$ be a solid weakly positive implicative weak BCC-algebra. Then for $x,y\in B(a)$, $a\in I(G)$
and $\varphi(x)=0x$ we have
$$
\arraycolsep=.5mm\begin{array}{rl}
 (xy\cdot y(0\cdot 0y))\cdot xy
&=(xy\cdot y\varphi^2(y))\cdot xy=(xy\cdot xy)\cdot
y\varphi^2(y)\\[3mm]
&=0\cdot
y\varphi^2(y)=\varphi(y)\varphi^3(y)=\varphi(y)\varphi(y)=0.
\end{array}
$$

Hence
\begin{equation} \label{2.5.8}
xy\cdot y(0\cdot 0y) \leqslant xy. 
\end{equation}

On the other hand,
$$\rule{10mm}{0mm}
\arraycolsep=.5mm\begin{array}{rl} xy\cdot (xy\cdot y(0\cdot
0y))&=xy\cdot (xy\cdot
y\varphi^2(y))\\[3mm]
&= ((xy\cdot y)\cdot 0y)\cdot (xy\cdot y\varphi^2(y)) \rule{21mm}{0mm} {\rm by } \ \eqref{2.5.7}\\[3mm]
&= (xy\cdot y)\varphi(y)\cdot (xy\cdot y\varphi^2(y))\\[3mm]
&= (xy\cdot y)(xy\cdot y\varphi^2(y))\cdot\varphi(y) \\[3mm]
&= (xy\cdot (xy\cdot y\varphi^2(y)))y\cdot\varphi(y),
\end{array}
$$
because, according to Lemma \ref{L-23}, we have $xy\cdot
y,\,\varphi(y)\in B(0a)$ and $xy,\,xy\cdot y\varphi^2(y)\in B(0)$.
Since in this case also $\,y\varphi^2(y)\in B(0)$, therefore
$$
(xy\cdot (xy\cdot y\varphi^2(y)))\cdot y\varphi^2(y)=(xy\cdot
y\varphi^2(y))\cdot (xy\cdot y\varphi^2(y))=0.
$$
Thus
$$
(xy\cdot (xy\cdot y\varphi^2(y)))\leqslant y\varphi^2(y),
$$
which, by \eqref{e2}, implies
$$
((xy\cdot (xy\cdot y\varphi^2(y))))y\cdot\varphi(y)\leqslant
(y\varphi^2(y))y\cdot\varphi(y).
$$
Hence
$$
\arraycolsep=.5mm\begin{array}{rl} xy\cdot (xy\cdot y(0\cdot
0y))&=(xy\cdot (xy\cdot y\varphi^2(y)))y\cdot\varphi(y)\\[3mm]
&\leqslant (y\varphi^2(y))y\cdot\varphi(y)\\[3mm]
&= (yy\cdot \varphi^2(y))\cdot\varphi(y)\\[3mm]
&=(0\cdot\varphi^2(y))\cdot\varphi(y)\\[3mm]
&= \varphi^3(y)\cdot\varphi(y)=0\\
\end{array}
$$
because $\,\varphi^3(y)=\varphi(y)\,$ by Theorem \ref{T-fi}.

This proves
\begin{equation} \label{2.5.9}
xy \leqslant xy\cdot y(0\cdot 0y)
\end{equation}
Combining \eqref{2.5.8} and \eqref{2.5.9} we get
\begin{displaymath}
xy = xy\cdot y(0\cdot 0y).
\end{displaymath}
So, $G$ is a solid branchwise $\varphi$-implicative weak
BCC-algebra.
\end{proof}

The converse of the Theorem \ref{T-510} is not true.

\begin{example}\label{Ex-256}\rm
It is not difficult to see that the following weak BCC-algebra
is proper and solid.

\begin{center}$
\begin{array}{c|ccccc} \cdot&0&1&2&3&4\\ \hline\rule{0pt}{11pt}
0&0&0&0&3&3\\
1&1&0&0&3&3\\
2&2&2&0&4&4\\
3&3&3&3&0&0\\
4&4&4&3&1&0
\end{array}$
\end{center}

\noindent It is branchwise $\varphi$-implicative but not weakly
positive implicative since $4\cdot 3=1$ and $((4\cdot 3)\cdot 3)\cdot (0\cdot 3)=0$.
\hfill$\Box{}$
\end{example}

\begin{theorem} \label{T-511}
A solid weak BCC-algebra is branchwise implicative if and only if
it is branchwise $\varphi$-implicative and branchwise commutative.
\end{theorem}
\begin{proof}
Let $G$ be a branchwise implicative solid weak BCC-algebra. Then
$x=x\cdot yx$ for $x,y\in B(a)$. Consequently,
\begin{equation}\label{xx}
x\cdot xy = (x\cdot yx)\cdot xy.
\end{equation}

Since $x\cdot yx,y\cdot yx\in B(a)$, we have
$$
(x\cdot xy)\cdot (y\cdot yx)\stackrel{\eqref{xx}}{=}((x\cdot
yx)\cdot xy)\cdot (y\cdot yx)\stackrel{\eqref{e-sol}}{=}(x\cdot
yx)(y\cdot yx)\cdot xy=0,
$$
by $(i)$. Hence $x\cdot xy\leqslant y\cdot yx$. Thus $x\cdot xy =
y\cdot yx$, which shows that $G$ is branchwise commutative.

Next, we obtain
$$
\arraycolsep=.5mm\begin{array}{rl}(xy\cdot y\varphi^2(y))\cdot
xy&=
(xy\cdot xy)\cdot y\varphi^2(y) = 0\cdot y\varphi^2(y)\\[3mm]
&=0y\cdot
0\varphi^2(y)=\varphi(y)\cdot\varphi^3(y)=\varphi(y)\cdot\varphi(y)=
0,
\end{array}
$$
because $\varphi$ is an endomorphism (Proposition \ref{P-211})
such that $\varphi^3=\varphi$ (Theorem \ref{T-fi}). Thus
\begin{equation} \label{2.5.10}
xy\cdot y\varphi^2(y)\leqslant xy.
\end{equation}

Moreover, from the the fact that a weak BCC-algebra $G$ is
branchwise commutative and $xy,\,y\varphi^2(y)\in B(0)$, we obtain
$$
\rule{15mm}{0mm}\arraycolsep=.5mm\begin{array}{rl}xy\cdot (xy\cdot
y\varphi^2(y))&= y\varphi^2(y)\cdot (y\varphi^2(y)\cdot xy)\\[3mm]
&= y\varphi^2(y)\cdot (y\cdot xy)\varphi^2(y)\rule{13mm}{0mm} {\rm since } \ \ \varphi^2(y)\in B(a)\\[3mm]
&= y\varphi^2(y)\cdot y\varphi^2(y)\hfill {\rm since } \
\ y\cdot xy=y\\[3mm]
&=0.
\end{array}
$$
Hence
\begin{equation} \label{2.5.11}
xy \leqslant xy\cdot y\varphi^2(y).
\end{equation}
Comparing \eqref{2.5.10} and \eqref{2.5.11} we get $xy = yx\cdot
y\varphi^2(y)$, so this weak BCC-algebra is $\varphi$-implicative.

Conversely, let a solid weak BCC-algebra $G$ be branchwise
$\varphi$-implicative and branchwise commutative. Then $x\cdot
yx\in B(a)$ for any $x,y\in B(a)$. Hence
$$
x(x\cdot yx)\cdot yx=(x\cdot yx)(x\cdot yx)=0.
$$
Consequently,
$$
x(x\cdot yx)=x(x\cdot yx)\cdot 0=x(x\cdot yx)\cdot (x(x\cdot
yx)\cdot yx).
$$
But $yx$ and $x(x\cdot yx)$ are in $B(0)$ and $G$ is branchwise
commutative, so we also have
$$
x(x\cdot yx)\cdot (x(x\cdot yx)\cdot yx)=yx\cdot (yx\cdot x(x\cdot
yx)).
$$

Thus
$$
x(x\cdot yx) = yx\cdot (yx\cdot x(x\cdot yx)).
$$

Since, by Lemma \ref{L-23}, elements $yx$, $yx\cdot x\varphi^2(x)$
and $yx\cdot x(x\cdot yx)$ are in $B(0)$, from the above, in view
of $\varphi$-implicativity of $G$ and Proposition \ref{P-212}, we
obtain
$$
\arraycolsep=.5mm\begin{array}{rl} x(x\cdot yx)
&=(yx\cdot x\varphi^2(x))\cdot (yx\cdot x(x\cdot yx))\\[3mm]
& \leqslant x(x\cdot yx)\cdot x\varphi^2(x)\leqslant
\varphi^2(x)(x\cdot yx),
\end{array}
$$
because $x(x\cdot yx),\,x\varphi^2(x)\in B(0)$.

Moreover, from $\varphi^2(x)\in B(a)$ we get
$a\leqslant\varphi^2(x)$, which, by Theorem \ref{T-fi}, implies
$a=\varphi^2(a)=\varphi^4(x)=\varphi^2(x)$. Thus
$$
x(x\cdot yx)\leqslant \varphi^2(x)(x\cdot yx)=a(x\cdot yx)=0,
$$
because $\,x\cdot yx\in B(a)$. Hence $\,x\leqslant x\cdot yx$.

On the other hand, $(x\cdot yx)x=xx\cdot yx=0\cdot yx=0$, which
together with the previous inequality gives $x\cdot yx=x$.

This completes the proof.
\end{proof}

\section{Weak BCC-algebras with condition (S)}

BCK-algebras with condition $(S)$ were introduced by K. Is\'eki in
\cite{I'77} and next generalized to BCI-algebras. Later such
algebras were extensively studied by several authors from
different points of view. Today BCK-algebras with condition $(S)$
are an important class of BCK-algebras.

Below we extend this concept to the case of weak BCC-algebras and
prove basic properties of these algebras.

For given two elements $x$ and $y$ of a weak BCC-algebra $G$ we
consider the set
$$A(x,y)=\{p\in G:px\leqslant y\}=\{p\in G: px\cdot y=0\}.
$$

We start with the following simple lemma.

\begin{lemma}\label{L-71}
Let $G$ be a weak BCC-algebra. Then for $x,y,z,u\in G$ we have
\begin{enumerate}
\item \ $A(0,x)=A(x,0)$,
\item \ $0\in A(x,y)\Longleftrightarrow 0\in A(y,x)$,
\item \ $x\in A(x,y)\Longleftrightarrow y\in B(0)$,
\item \ $x\in B(0)\Longrightarrow y\in A(x,y)$,
\item \ $A(x,y)\subset A(u,y)$ for $x\leqslant u$,
\item \ $A(x,y)\subset A(x,z)$ for $y\leqslant z$,
\item \ $u\leqslant z$, $z\in A(x,y)\Longrightarrow u\in A(x,y)$,
\item \ $A(x,y)=A(y,x)$ if $G$ is a BCI-algebra. \hfill$\Box{}$
\end{enumerate}
\end{lemma}

Example \ref{Ex-256} shows that in general $A(x,y)\ne A(y,x)$.
Indeed, in a weak BCC-algebra defined in this example $A(3,4)\ne
A(4,3)$, $3$, $4$ are not in $A(3,4)$ and $3\in A(2,3)$ does not
imply $2\in B(0)$.

\begin{proposition}\label{P-26.1}
Let $G$ be a solid weak BCC-algebra. If $x\in B(a)$, $y\in B(b)$,
then $A(x,y)$ is a non-empty subset of the branch $B(a\cdot
0b)$.
\end{proposition}
\begin{proof} Let $x\in B(a)$, $y\in B(b)$. Then $\varphi^2(x)\in I(G)\cap B(a)$ and $\varphi(y)=\varphi(b)$,
by Theorem \ref{T-fi}. Hence $\varphi^2(x)=a$. Moreover, since
$\varphi$ is an endomorphism (Proposition \ref{P-211}), we have
$$
s=0(0x\cdot y)=\varphi(\varphi(x)\cdot y)=
\varphi^2(x)\cdot\varphi(y)=a\cdot\varphi(b)= a\cdot 0b.
$$
Therefore,
$$
sx\cdot y=(a\cdot 0b)x\cdot y=(ax\cdot 0b)y=(0\cdot 0b)y=by=0,
$$
shows that $s\in A(x,y)$. Thus the set $A(x,y)$ is non-empty.

Let $p$ be an arbitrary element of $A(x,y)$. Then $px$ and $y$ are
in the same branch. Consequently
$$
s=0(0x\cdot y)=0\cdot (pp\cdot x)y=0\cdot (px\cdot p)y=0\cdot
(px\cdot y)p=0\cdot 0p=\varphi^2(p)\leqslant p,
$$
by Theorem \ref{T-fi}.

So, $s$ is the least element of $A(x,y)$ and $A(x,y)\subset B(s)$.
\end{proof}

\begin{definition}\label{D-26.3} \rm We say that a solid weak BCC-algebra $G$ is {\it with
condition $(S)$}, if each its subset $A(x,y)$ has the greatest
element. The greatest of $A(x,y)$ will be denoted by $x\circ y$.
\end{definition}

\begin{example}\label{E-26.4}\rm
Let $(G;*,0)$ be an abelian group. Then $(G;\cdot,0)$ with the
operation $xy=x*y^{-1}$ is a solid weak BCC-algebra in which each
branch has only one element. Thus $A(x,y)\subset B(x\cdot
0y)=\{x\cdot 0y\}$. Consequently $x\circ y=x\cdot 0y=x*y$.
\hfill$\Box{}$
\end{example}

\begin{example}\label{E-26.5}\rm
Each finite solid weak BCC-algebra decomposed into linearly
ordered branches is with condition $(S)$ since each set $A(x,y)$
is a finite subset of some linearly ordered branch. \hfill$\Box{}$
\end{example}

\begin{example}\label{E-26.6}\rm A solid weak BCC-algebra defined
by the table:
$$
\begin{array}{c|cccccc} \cdot&0&a&b&c&d\\\hline
0&0&0&b&b&b\rule{0mm}{4mm}\\
a&a&0&b&b&b\\
b&b&b&0&0&0\\
c&c&b&a&0&a\\
d&d&b&a&a&0
\end{array}
$$
is not with condition $(S)$ since $A(a,b)=\{b,c,d\}=B(b)$ has no
greatest element. \hfill$\Box{}$
\end{example}

Since in a solid weak BCC-algebra $G$ with condition $(S)$, for
each $x,y\in G$ the set $A(x,y)$ has the greatest element $x\circ
y$, so $\circ$ can be treated as a binary operation defined on
$G$ and $(G;\circ,0)$ can be considered as an algebra of type
$(2,0)$. Since in any case $A(x,0)=A(0,x)$, the groupoid
$(G;\circ,0)$ has the identity $0$. In the case of BCI-algebras
with condition $(S)$, $(G,\cdot,0)$ is a commutative semigroup
(cf. \cite{I'77}). For weak BCC-algebras it is not
true.
\begin{example}\label{E-26.7}\rm A weak BCC-algebra defined by the table:
$$
\begin{array}{c|cccc} \cdot&0&1&2&3\\\hline
0&0&0&2&2\rule{0mm}{4mm}\\
1&1&0&2&2\\
2&2&2&0&0\\
3&3&3&1&0\\
\end{array}\\
$$
is with condition $(S)$, but in this algebra $1\circ
2\ne 2\circ 1$ and $(2\circ 2)\circ 2\ne 2\circ (2\circ 2)$. \hfill$\Box{}$
\end{example}

For some BCI-algebras (described in \cite{Du'86} and \cite{Du'88})
$(G;\circ,0)$ is an abelian group. A similar situation takes place
in the case of weak BCC-algebras. To prove this fact we need the
following lemma.

\begin{lemma}\label{L-26.7}
In weak BCC-algebras with condition $(S)$
$$
x\leqslant y\Longrightarrow x\circ z\leqslant y\circ z
$$
for all $x,y,z\in G$.
\end{lemma}
\begin{proof}
If $x\leqslant y$, then also $(x\circ z)y\leqslant (x\circ z)x$,
by \eqref{e3}. But, according to the definition, $(x\circ
z)x\leqslant z$. Thus $(x\circ z)y\leqslant z$, which implies
$x\circ z\leqslant y\circ z$ because $y\circ z$ is the greatest
element satisfying the inequality $py\leqslant z$.
\end{proof}

\begin{theorem}\label{T-26.8} Let $(G;\cdot,0)$ be a weak BCC-algebra with condition $(S)$.
Then $(G;\circ,0)$ is a group if and only if $(G;\cdot,0)$ is
group-like.
\end{theorem}
\begin{proof} Let $(G;\circ,0)$ be a group. Consider an arbitrary element $x\in
B(0)$. Denote by $x^{-1}$ the inverse element of $x$ in a group
$(G;\circ,0)$. By Lemma \ref{L-26.7} from $0\leqslant x$ it
follows $x^{-1}=0\circ x^{-1}\leqslant x\circ x^{-1}=0$. Hence
$x^{-1}=0$. Thus $B(0)=\{0\}$. This, by Proposition 3.15 in
\cite{DKB}, shows that a weak BCC-algebra $(G;\cdot,0)$ is
group-like.

Conversely, if a weak BCC-algebra $(G;\cdot,0)$ is group-like,
then each its branch has only one element. Hence $px\leqslant y$
means $px=y$, i.e., $p*x^{-1}=y$ in the corresponding group
$(G;*,0)$. Thus $p=y*x$ is uniquely determined by $x,y\in G$.
Therefore $x\circ y=y*x$. So, $(G;\circ,0)$ is a group.
\end{proof}
\begin{corollary}\label{C-26.9}
Let $(G;\cdot,0)$ be a weak BCC-algebra with condition $(S)$. Then
$(G;\circ,0)$ is an abelian group if and only if $(G;\cdot,0)$ is
a group-like BCI-algebra.
\end{corollary}
\begin{proof}
Indeed, $(G,\circ,0)$ is abelian if and only $(G;*,0)$ is abelian,
that is, if and only if $xy\cdot
z=x*y^{-1}*z^{-1}=x*z^{-1}*y^{-1}=xz\cdot y$.
\end{proof}

\begin{proposition}\label{P-26.10}
In a solid weak BCC-algebra with condition $(S)$ we have
\begin{equation}\label{2.6.1}
xy\cdot z=x(y\circ z)
\end{equation}
for $x,y$ belonging to the same branch and $z\in B(0)$.
\end{proposition}
\begin{proof} Let $x,y\in B(a)$, $z\in B(0)$. Then
$x(xy\cdot z)\cdot y=xy\cdot (xy\cdot z)\leqslant z$ implies
$x(xy\cdot z)\cdot y\leqslant z$. Thus $x(xy\cdot z)\in A(y,z)$.
Hence $x(xy\cdot z)\leqslant y\circ z$ and $x(y\circ z)\leqslant
x(x(xy\cdot z))$, by \eqref{e3}. Further, since $x(xy\cdot z)\in
B(a)$, we have
$$
x(x(xy\cdot z))\cdot (xy\cdot z)=x(xy\cdot z)\cdot x(xy\cdot z)=0.
$$
So, $x(x(xy\cdot z))\leqslant xy\cdot z$. Consequently,
$$
x(y\circ z)\leqslant x(x(xy\cdot z))\leqslant xy\cdot z.
$$

On the other hand, $y\circ z\in A(y,z)\subset B(a)$, by
Proposition \ref{P-26.1}. This, by Lemma \ref{L-21}, gives
$x(y\circ z)\in B(0)$. Thus,
$$
(xy\cdot x(y\circ z))\cdot (y\circ z)y=(xy\cdot (y\circ z)y)\cdot
x(y\circ z)\leqslant x(y\circ z)\cdot x(y\circ z)=0.
$$
Consequently, $xy\cdot x(y\circ z)\leqslant (y\circ z)y\leqslant
z$.

Therefore,
$$
0=(xy\cdot x(y\circ z))\cdot z=(xy\cdot z)\cdot x(y\circ z),
$$
which implies $(xy\cdot z)\leqslant x(y\circ z)$. Hence $ xy\cdot
z=x(y\circ z)$.
\end{proof}

\begin{corollary}\label{C-26.11}
In a solid weak BCC-algebra with condition $(S)$ the branch $B(0)$
satisfies the identity $\eqref{2.6.1}$. \hfill$\Box{}$
\end{corollary}

\begin{corollary}\label{C-26.12}
A BCI-algebra with condition $(S)$ satisfies the identity
$\eqref{2.6.1}$. \hfill$\Box{}$
\end{corollary}

\begin{theorem}\label{T-26.13}
A solid weak BCC-algebra with condition $(S)$ is restricted if and
only if some its branch is restricted.
\end{theorem}
\begin{proof}
Assume that some branch, for example $B(a)$, is restricted and
$1_a$ is the greatest element of $B(a)$. Then $xb\in B(0)$ for
every $x\in B(b)$ and an arbitrary $b\in I(G)$. Thus $xb\cdot
0a\in B(0\cdot 0a)=B(a)$. Hence $xb\cdot 0a\leqslant 1_a$, i.e.,
$\,xb\leqslant 0a\circ 1_a, $ according to the definition of
$\,0a\circ 1_a$. Consequently,
\begin{equation}\label{2.6.2}
xb\cdot (0a\circ 1_a)=0 \ \ \ {\rm and } \ \ \ (0a\circ 1_a)\cdot
0a\leqslant 1_a.
\end{equation}
Hence $(0a\circ 1_a)\cdot 0a\in B(a)$, i.e., $a\leqslant (0a\circ
1_a)\cdot 0a$. Since $\varphi(x)=0x$ is an endomorphism
(Proposition \ref{P-211}), from the last inequality, applying
Theorem \ref{T-fi} $(2)$, we obtain
$$
0a=0((0a\circ 1_a)\cdot 0a)=0(0a\circ 1_a)\cdot (0\cdot
0a)=0(0a\circ 1_a)\cdot a.
$$
Therefore
$$
0=(0(0a\circ 1_a)\cdot a)\cdot 0a\leqslant 0(0a\circ 1_a)\cdot
0=0(0a\circ 1_a),
$$
by $(i')$. This, by Theorem \ref{T-fi} $(1)$, gives $0=0\cdot
0=0\cdot 0(0a\circ 1_a)\leqslant 0a\circ 1_a$. So, $\,0a\circ
1_a\in B(0)$.

Now let $m=b\circ (0a\circ 1_a)$. Then for every $x\in B(b)$,
according to Proposition \ref{P-26.10} and \eqref{2.6.2}, we have
$$
xm=x(b\circ (0a\circ 1_a))=xb\cdot (0a\circ 1_a)=0,
$$
which implies $x\leqslant m$. Therefore $m$ is the greatest
element of the branch $B(b)$.

The converse statement is obvious.
\end{proof}

\section{Conclusions}

In the study of various types of algebras inspired by logic a very important role plays the identity 
$xy\cdot z=xz\cdot y$ which is not satisfied in weak-BCC-algebras. In this paper we described weak-BCC-algebras satisfying this identity in the case when elements $x$ and $y$ (or $x$ and $z$) are in the same branch. Using the method presented above we can obtain results which are similar to results proved earlier for BCI-algebras. Our method based on the restriction of the verification of various properties to their verification only to elements belonging to the same branch makes it possible to study these properties for the wider class of algebras.

Further results on solid weak BCC-algebras one can find in
\cite{BT} and \cite{Thom}. In the first paper some important identities satisfied in weak BCC-algebras are described; in the second -- $f$-derivations of weak
BCC-algebras. Since obtained results are very similar to those proved for example for BCI-algebras seems that verification of various properties can be reduced to verification in
branches only which is very important for computer verification.

\end{document}